\documentclass[12pt,reqno]{amsart} 
\usepackage{amssymb}
\usepackage{graphicx} 
\usepackage{pdfsync}
\input xy
\xyoption{all}
\title[Aschbacher--O'Nan--Scott theorem for linear groups]{An Aschbacher--O'Nan--Scott theorem for countable linear groups}
\author{Tsachik Gelander and Yair Glasner}
\thanks{T.G. acknowledges support of the European Research Council (ERC)/ grant agreement 203418 and the ISF grant 1345/07. Y.G. acknowledges support of the ISF grant 441/11. }
\email{yairgl@math.bgu.ac.il}
\email{gelander@math.huji.ac.il}

\newtheorem*{theorem}{Theorem}
\newtheorem*{lemma}{Lemma}

 \theoremstyle{definition}

\newtheorem*{definition}{Definition}
\newtheorem*{remark}{Remark}

\newtheorem*{example}{Example}

\newcommand{\N}{{\mathbf{N}}}
\newcommand{\Z}{{\mathbf{Z}}}

\newcommand{\C}{{\mathbf{C}}}
\newcommand{\Q}{{\mathbf{Q}}}
\newcommand{\F}{{\mathbf{F}}}

\newcommand{\G}{{\mathbb{G}}}

\renewcommand{\P}{{\mathbf{P}}}

\newcommand{\ov}{{\overline{v}}}
\newcommand{\oH}{{\overline{H}}}
\newcommand{\Ac}{{\mathcal{A}}}
\newcommand{\Rc}{{\mathcal{R}}}

\newcommand{\action}{\curvearrowright}
\newcommand{\arrow}{\rightarrow}

\newcommand{\Aut}{{\operatorname{Aut}}}

\newcommand{\Core}{{\operatorname{Core}}}

\newcommand{\PSL}{{\operatorname{PSL}}}
\newcommand{\SL}{{\operatorname{SL}}}

\newcommand{\GL}{{\operatorname{GL}}}
\newcommand{\PGL}{{\operatorname{PGL}}}

\newcommand{\Inn}{\mathrm{Inn}}

\newcommand{\trivgp}{\langle e \rangle}
\newcommand{\defeq}{\stackrel{\operatorname{def}}{=}}

\begin{document}
\bibliographystyle{alpha}
\begin{abstract}
The purpose of this note is to extend the classical Aschbacher--O'Nan--Scott theorem for finite groups to the class of countable linear groups.
This relies on the analysis of primitive actions carried out in \cite{GG:Primitive}. Unlike the situation for finite groups, we show here that the number of primitive actions depends on the type: linear groups of almost simple type admit infinitely (and in fact unaccountably) many primitive actions, while affine and diagonal groups admit only one. The abundance of primitive permutation representations is particularly interesting for rigid groups such as simple and arithmetic ones.

\end{abstract}

\maketitle

\section{Introduction}

A group action $\Gamma \action \Omega$ is called {\it primitive} if there is no $\Gamma$-invariant equivalence relation on the set $\Omega$. Equivalently, primitive actions are transitive, of the form $\Gamma \action \Gamma/ \Delta$, where $\Delta$ is a maximal subgroup. One says that a group $\Gamma$ is a {\it primitive group} if it admits a faithful primitive action on a set. Whenever we use the word {\it{countable}} in this note we mean infinite and countable. 


The classical Aschbacher--O'Nan--Scott theorem (hereafter AOS theorem) describes the structure of finite primitive groups. In what follows we will always refer to the version of this theorem as it appears in \cite[Section 4.8, page 137]{DM:Permutation_Groups} where the finite primitive groups are sorted into five distinct categories; which we refer to as {\it{AOS categories}}. The classification there strongly depends on the {\it{socle}} of the group which is just the product of its minimal normal subgroups. 

When analyzing countable primitive linear groups one quickly finds infinite counterparts to many of the AOS categories. We describe infinite linear analogs for primitive groups of affine and Diagonal types in Sections \ref{sec:affine} and \ref{sec:diagonal} respectively.  It is not difficult to come up with analogues of the other AOS categories. For example, for any field $F$, the almost simple group $\PSL_2(F)$ acts $3$-transitively on the corresponding projective line $\P^1 F$. However, when attempting to classify primitive groups one soon encounters actions that do not resemble any of the finite AOS categories. For example the free group $F_2$ has a trivial socle (as it has no minimal normal subgroups), but it was shown by McDonough \cite{MD_free_HT} that this group admits a highly transitive action, i.e. an action that is $k$-transitive for every $k$. 

In \cite{GG:Primitive} we find the right generalization of the concept of almost simple groups, in the setting of countable linear groups. Instead of considering simplicity of the group itself we say that a group is of {\it{almost simple type}} if it admits a faithful linear representation whose Zariski closure is simple or close to being simple (see precise definition in Section \ref{sec:almost_simple}). It is shown in \cite[Theorem 1.9]{GG:Primitive} that any countable primitive nontorsion linear group falls into exactly one of the three categories: primitive of {\it{affine type}}, primitive of {\it{diagonal type}} or primitive of {\it{almost simple type}}. While primitive groups of affine type directly generalize the corresponding finite AOS category, primitive groups of diagonal type actually generalize only the special situation where $H$ is a product of two minimal normal subgroups. The latter category, primitive groups of almost simple type, contains all other primitive groups. It encompasses simple groups like $\PSL_2(F)$ together with many groups with trivial socle such as free groups and arithmetic groups (such as $\PSL_n(\Z), \ \ n \ge 2$).  A fact that attracted some criticism after the publication of \cite{GG:Primitive} was that the category of infinite groups of almost simple type contains also groups that seem to be direct infinite generalizations of groups of product or diagonal type, such as for example $T^m \rtimes S_m$ where $T$ is any infinite simple linear group. 

Our main result - Theorem \ref{thm:main} - reinforces the analogy between our classification of countable linear primitive groups with the finite Aschbacher-O'Nan-Scott theorem. We show that, while the primitive action is uniquely determined by the algebraic structure of the group for countable linear groups of affine and diagonal types; groups of almost simple type admit uncountably many nonisomorphic faithful primitive actions. For example the group $\PSL_2(\Q)$ admits many faithful primitive actions in addition to its well known 3-transitive action on the projective line $\P^1 \Q$. A similar situation holds for groups of almost simple type, that might be constructed in the same manner as groups of product or diagonal AOS types. 

This paper, as well as our previous work \cite{GG:Primitive}, were inspired by the beautiful paper of Margulis and Soifer \cite{MS:Maximal}.

\section{Types of primitive actions}

 \subsection{Primitive groups of affine type} 
 \label{sec:affine}
\begin{definition} \label{def:affine} Let $M$ be a countable vector space over a prime field, namely either $M = \F_p^{\aleph_0}$ or $M = \Q^{n}$ with $n \in \N \cup \{\infty\}$. Let $\Delta \leq \GL (M)$ be such that there are no nontrivial $\Delta$-invariant subgroups of  $M$. 

The action of $M$ on itself by left translations combined with the action of $\Delta$ on $M$ by conjugation yields an action of the semidirect product $\Gamma := \Delta \ltimes M \action M$. We refer to this as the {\it affine} action of $\Gamma$ on $M$. It follows from the condition on $\Delta$ that the affine action is primitive. We say in this case that the permutation group $\Gamma \action M$ is {\it primitive of affine type}.
\end{definition}

\begin{example} \label{ex:affine}
The natural two transitive action $\Q^{*} \ltimes \Q \action \Q$ is primitive of affine type. More generally one can consider the group $GL_n(F) \ltimes F^n$ with its affine action on $F^n$, for any countable field $F$. 
\end{example}

\subsection{Primitive groups of diagonal type}
\label{sec:diagonal}
\begin{definition} \label{def:diagonal}
Let $M$ be a nonabelian characteristically simple group, and $\Delta \leq \Aut(M)$. Assume that $\Inn(M) \leq \Delta$ and that there are no nontrivial $\Delta$-invariant subgroups of $M$. Just as in the affine case our condition on $\Delta$ ensures that the affine action $\Gamma = \Delta \ltimes M \action M$ is primitive. We then say that the permutation group $\Gamma$ is {\it primitive of diagonal type}.
\end{definition}

\begin{remark} \label{rem:diagonal} A group $\Gamma$ of diagonal type as above contains another normal subgroup which is isomorphic to  and commutes with $M$. This group is
$$M' \defeq \{\iota(m^{-1}) m \ | \ m \in M \},$$ 
where $\iota: M \arrow \Inn(M) < \Delta$ is the natural injection. Thus the action of $M  \times M'$ on $M$ is given by $(m,m')\cdot x = m x (m')^{-1}$, so that this action can be identified with the action of $M \times M'$ on the cosets of the diagonal subgroup $\{(m,m) \ | \ m \in M\}$. This is where the terminology {\it{ diagonal groups}} comes from. 
\end{remark}
\begin{example} \label{ex:diagonal}
The direct power $M = S^n$ of a simple group is always characteristically simple. While in the finite case every characteristically simple group is of this form, in the infinite case there are more examples. For instance $M = \PSL_n(\F[x])$ is characteristically simple by \cite{Wilson:CS} and hence $\Aut(M) \ltimes M$ is primitive of diagonal type.  
\end{example}

\subsection{Groups with an almost simple Zariski closure}
\label{sec:almost_simple}
\begin{definition} \label{def:AS}
We say that a linear group $\Gamma$ is {\it primitive of almost simple type} if there exists a faithful linear representation $\rho: \Gamma \arrow \GL_n(K)$ over an algebraically closed field $K$ for which the identity connected component of the Zariski closure $$G \defeq \left(\overline{\rho(\Gamma)}^{Z}\right)^{0} = H \times H \times \ldots \times H$$ is a product of isomorphic, simple center-free algebraic groups and the action of $\Gamma$ by conjugation is transitive on these simple factors. 
\end{definition}

\begin{example} \label{ex:AS}
$\SL_2(\C)$ contains a Zariski dense copy of every countable free group. More generally, by the Borel density theorem, every lattice in a connected noncompact simple Lie group has a simple Zariski closure. 
\end{example}
\subsection{The classification of countable primitive linear groups, from \cite{GG:Primitive} }\label{thm:GG}
Let $\Gamma$ be a countable linear group. If the ground field has positive characteristic, assume further that $\Gamma$ is not torsion (i.e. not all elements of $\Gamma$ have finite order). Then $\Gamma$ admits a faithful primitive action on a countable set if and only if it falls into one of the following, mutually exclusive, categories:
\begin{enumerate}
\item $\Gamma$ is primitive of affine type,
\item $\Gamma$ is primitive of diagonal type,
\item $\Gamma$ is primitive of almost simple type. 
\end{enumerate}


\section{The main theorem}
\label{sec:main_thm}
\begin{definition}
An action $\Gamma \action \Omega$ of a group on a set is called {\it{quasiprimitive}} if every normal subgroup $N \lhd \Gamma$ acts either trivially or transitively on $\Omega$. 
\end{definition}
Every primitive group action is quasi-primitive because the orbits of a normal subgroup are equivalence classes for a $\Gamma$-invariant equivalence relation. 
\begin{theorem} \label{thm:main}
Let $\Gamma$ be a countable primitive linear group as in Theorem \ref{thm:GG}, then the following dichotomy holds:
\begin{enumerate} 
\item If {\bf{$\bf{\Gamma}$ is primitive of affine or diagonal type}}, then it admits a unique (up to isomorphism of actions) faithful quasiprimitive action. Moreover this action is primitive.  
\item  If {\bf{$\bf{\Gamma}$ is primitive of almost simple type}}, then it admits uncountably many nonisomorphic faithful primitive actions on a countable set.  
\end{enumerate}
\end{theorem}
\begin{remark} \label{rem:cont_hyp}
We establish the existence of uncountably many (i.e. $\gneqq \aleph_0$) nonisomorphic primitive actions. It is a natural question whether one can construct $2^{\aleph_0}$ nonisomorphic actions without appealing to the continuum hypothesis. 
\end{remark}
\begin{remark}
The faithfulness assumption is necessary in the theorem above. Take for example the group $\SL_2(\Q) \ltimes \Q^2$. This is a typical group of affine type and hence admits a unique faithful primitive action. But it maps onto the group $\PSL_2(\Q)$ which is clearly of almost simple type and hence admits uncountably many nonisomorphic primitive actions. 
\end{remark}

\begin{proof}[Proof of the Theorem] 
We use the notations from \cite{BG:Dense_Free,BG:Topological_Tits,GG:Primitive}. In particular by ``the canonical" attracting point or repelling hyperplane of a proximal projective transformation $g\in \PGL_n(k)$, we mean these fixed point $\overline{v}_{g}$ and fixed hyperplane $\overline{H}_{g}$ obtained in \cite[Lemma 3.2]{BG:Topological_Tits}.

We shall start with few reductions. Let $\Gamma$ be a countable primitive linear group. If $\Gamma$ is linear over a field of characteristic $p > 0$ assume further that $\Gamma$ is not torsion. For $\Gamma$ of affine or diagonal type, it is proved in \cite[Proposition A.1(2)]{GG:Primitive} that the given primitive action of $\Gamma$ is the unique faithful quasiprimitive action of this group. Below we will assume that $\Gamma$ is of almost simple type, i.e. that it comes with a linear representation as in Section \ref{sec:almost_simple}  and prove that $\Gamma$ admits uncountably many nonisomorphic faithful primitive actions. 

Recall that a group action $\Gamma \curvearrowright \Gamma/\Delta$ is  faithful if and only if 
$$
 \Core_{\Gamma}(\Delta) := \cap_{\gamma \in \Gamma} \Delta ^{\gamma} = \trivgp.
$$ 
So in group theoretic terms we have to show that $\Gamma$ contains uncountably many nonconjugate maximal subgroups of infinite index and trivial core. In practice we need never worry about the conjugacy between maximal subgroups because, since $\Gamma$ is countable, there are at most countably many different maximal subgroups conjugate to any given one.

\medskip

Let $K$ be an arbitrary field and $\Gamma \le \GL_n(K)$ a countable group for which the connected component $\G^{0}$ of $\G = \overline{\Gamma}^{Z}$ is a power of simple $K$ algebraic group and the action of $\Gamma$ on $\G^{0}$ by conjugation is faithful and permutes the simple factors of $\G^{0}$ transitively. In \cite{GG:Primitive} we construct
\begin{itemize}
\item a complete valuation field $k$ (which is a local field in case $\Gamma$ happens to be finitely generated),
\item a strongly irreducible algebraic projective representation $$\rho : \G(k) \arrow \PGL_n(k),$$ defined over a local subfield $k' < k$, and
\item elements $\{x,y,\delta_1,\delta_2, \ldots, \eta_1,\eta_2,\ldots\} \subset \Gamma$.
\end{itemize}
Such that the following properties are satisfied
\begin{enumerate}
\item[(i)] $\{\delta_1,\delta_2, \ldots, \eta_1,\eta_2,\ldots\}$ are very proximal elements, forming a ping-pong tuple with respect to the action on $\P^{n-1}(k)$. Furthermore each of these elements satisfies the conditions of \cite[Lemma 3.2]{BG:Topological_Tits}. 
\item[(ii)] $\Delta \defeq \langle \delta_1,\delta_2, \ldots \rangle$ is {\it prodense} in $\Gamma$. By definition this means that $\Delta N = \Gamma,$ for every normal subgroup $\langle e \rangle \ne N \lhd \Gamma$. 
\item[(iii)] $\Delta$ has a nontrivial intersection with every double coset of $F \defeq \langle x,y \rangle$, $$\Delta \cap F \gamma F \ne \emptyset \ \ \forall \gamma \in \Gamma.$$
\end{enumerate}
Denote the attracting points and repelling hyperplanes associated with the very proximal element $\eta_i$ by $\ov_{\eta_i}, \ov_{\eta_i^{-1}},\oH_{\eta_i}, \oH_{\eta_i^{-1}}$, and the corresponding attracting and repelling neighborhoods by $\Ac(\eta_i)$, $\Ac(\eta_i^{-1})$, $\Rc(\eta_i)$, $\Rc(\eta_i^{-1})$.

\medskip

\begin{remark}
A reader willing to take the above construction from \cite{GG:Primitive} upon faith is not required to delve any further into the technical details of \cite{GG:Primitive} in order to understand the argument below. However in order to verify the validity of this construction we suggest the following ``road map" of \cite{GG:Primitive}:

For the construction of the elements $\{\delta_i\}$ in \cite{GG:Primitive} see Section 7.3, and for the construction of $x,y$ (referred to as $h_1,h_2$) see the final paragraph of Section 7 and the paragraph at the bottom of Page 1494 at Section 3.
The elements $\{\eta_i\}$ do not actually appear in \cite{GG:Primitive}. Nevertheless the inductive argument constructing the countably many ping-pong players $\{\delta_1,\delta_2,\ldots\}$ can equally well yield the additional elements $\{\eta_i\}$. Each element in its turn is required to satisfy some useful property as well as to play ping-pong with all the previous ones. We can vary the procedure slightly, dedicating the odd steps of the induction to this construction, while reserving the even steps for the construction of the elements $\eta_i$.
\end{remark}


We will make use of the following simple group theoretic lemma:

\begin{lemma} \label{lem:gp}
Let $\Gamma$ be a group, $M \lneqq \Gamma$ a proper subgroup, and assume that $\gamma, l \in \Gamma$ with $l \not \in M$. Then
$$\left \{\gamma l \gamma, \gamma^2 l \gamma \right\} \not \subset M ~~~~ {\text{ and }} ~~~~ \left \{\gamma l, \gamma^2 l \right\} \not \subset M$$
\end{lemma}
\begin{proof}
If both $\gamma l \gamma$ and $\gamma(\gamma l \gamma)$ belong to $M$ then also $\gamma \in M$ and hence $l = \gamma^{-1} (\gamma l \gamma) \gamma^{-1} \in M$, in contradiction to our assumption. The second statement's proof is identical.
\end{proof}

Assume, by way of contradiction that $\Gamma$ admits only countably or finitely many maximal subgroups of infinite index with trivial core. We enumerate all these subgroups 
$$M_1,M_2,M_3, \ldots$$
and conclude the theorem by constructing a new maximal subgroup $\Delta < M \lneqq \Gamma$, such that $M \ne M_i \ \ \forall i$. Note that as $\Delta$ is prodense so is every group containing it, and as such $M$ is of infinite index and has trivial core.

\subsubsection*{First step}
It is possible to choose elements $l_i  \in \Gamma \setminus M_i$ such that 
\begin{equation}
\label{eqn:pre_in}
l_i(\overline{v}_{\eta_i}) \not \in \overline{H}_{\eta_i} {\text{ and }} l_i(\overline{v}_{\eta_i^{-1}}) \not \in \overline{H}_{\eta_i^{-1}},
\end{equation}
Indeed let $l_i \in \Gamma \setminus M_i$ be any element. If Equation (\ref{eqn:pre_in}) does not hold choose an element $a \in \Gamma$ such that
$$ a l_i \overline{v}_{\eta_i} \not \in \overline{H}_{\eta_i} \quad 
a^2 l_i \overline{v}_{\eta_i} \not \in \overline{H}_{\eta_i} \quad
a l_i \overline{v}_{\eta_i^{-1}} \not \in \overline{H}_{\eta_i^{-1}} \quad 
a^2 l_i \overline{v}_{\eta_i^{-1}} \not \in \overline{H}_{\eta_i^{-1}}.$$
This is possible since $\Gamma$ acts strongly irreducibly on $\P^{1}(k)$. By the lemma at least one of the elements $al_i, a^2l_i$ does not belong to $M$. After replacing $l_i$ by that element Equation (\ref{eqn:pre_in}) holds. 

\subsubsection*{Second step}
We can require
\begin{equation}
\label{eqn:in}
l_i \Ac(\eta_i) \cap \Rc(\eta_i) =  l_i \Ac(\eta_i^{-1}) \cap \Rc(\eta_{i}^{-1}) = \emptyset.
\end{equation} 

Equation (\ref{eqn:in}) is achieved by replacing  $\eta_i$ by some positive proper power of it, thereby reducing the neighborhoods $\Ac(\eta_i)$, $\Rc(\eta_i)$, $\Ac(\eta_i^{-1})$, $\Rc(\eta_i^{-1})$. By \cite[Lemma 3.2]{BG:Topological_Tits} this dose not harm the properties (i,ii,iii) above.
Finally we appeal once more to the lemma, replacing $\eta_i$ by an element $\theta_i$ chosen from the set $\{\eta_i l_i \eta_i, \eta_i^2 l_i \eta_i\}$ such that $\theta_i \not \in M_i$. It is easy to see that these elements $\theta_i$ admit similar dynamics as the original $\eta_i$. For example:
\begin{eqnarray*}
\eta_i l_i \eta_i \left(\P^{n-1}(k) \setminus \Rc(\eta_i) \right) & \subset &
\eta_i l_i \left( \Ac(\eta_i) \right) \\
& \subset & \eta_i \left( \P^{n-1}(k) \setminus \Rc(\eta_i) \right) \subset \Ac(\eta_i)
\end{eqnarray*}

Now, since the attracting and repelling neighborhoods for $\theta_i$ are contained in the corresponding neighborhoods for $\eta_i$, the collection $\{\delta_1,\delta_2, \ldots, \theta_1, \theta_2 \ldots \}$ still forms a ping-pong tuple and is therefore independent. Let $\Sigma \defeq \langle \delta_1,\delta_2, \ldots, \theta_1, \theta_2 \ldots \rangle $ be the subgroup generated by these elements and let $\Sigma \le M \lneqq \Gamma$ be a subgroup of $\Gamma$, containing $\Sigma$ and maximal with respect to the property that it does not contain $F = \langle x,y \rangle$ as a subgroup. Such a group exists by a simple Zorn's lemma argument. Furthermore $M$ is a maximal subgroup of $\Gamma$, because any subgroup that strictly contains $M$ must contain $F$ together with a representative for every double coset $F \gamma F$, by Property (3) of $\Delta$. Since $\Delta$ is prodense so is the group $M > \Delta$ and hence $\Core_{\Gamma}M = \trivgp$. So $M$ is a maximal subgroup with trivial core. But for every $i \in \N$ we have $M \ne M_i$ because $\theta_i \in M \setminus M_i$. This is a contradiction to the fact that the collection $\{M_1,M_2,\ldots \}$ consists of all the maximal subgroups with trivial core inside $\Gamma$.

\end{proof}

\bibliography{../MyTexfiles/tex_utils/yair}
\end{document}